\documentclass[a4paper,12pt,twoside,reqno]{amsart}
\usepackage{amsmath}
\usepackage{amsthm}
\usepackage{amssymb}
\theoremstyle{plain}
\newtheorem{thm}{\indent\bf Theorem}[section]
\newtheorem{lem}[thm]{\indent\bf Lemma}
\newtheorem{prop}[thm]{\indent\bf Proposition}

\theoremstyle{definition}
\newtheorem{rem}{\indent\it Remark}[section]

\numberwithin{equation}{section}
\numberwithin{figure}{section}

\def \re {\mathrm{Re\,}}
\def \im {\mathrm{Im\,}}

\def \z {\mathfrak{z}}
\def \sn {\mathrm{sn}}
\def \cn {\mathrm{cn}}
\begin{document}
\title[Fifth Painlev\'e transcendents]
{Two error bounds of the elliptic asymptotics 
for the fifth Painlev\'e transcendents}
\author[Shun Shimomura]{Shun Shimomura} 
\address{Department of Mathematics, 
Keio University, 
3-14-1, Hiyoshi, Kohoku-ku,
Yokohama 223-8522 
Japan\quad
{\tt shimomur@math.keio.ac.jp}}
\date{}
\begin{abstract}
For the fifth Painlev\'e equation it is known that a general solution is
represented asymptotically by an elliptic function in cheese-like strips
near the point at infinity. 
We present an explicit asymptotic formula for the error term of this 
expression, which leads to an estimate for its magnitude  
as was conjectured. 
An analogous formula is obtained for the error term of
the correction function associated with the Lagrangian. 
\end{abstract}
\maketitle
\allowdisplaybreaks
\section{Introduction}\label{sc1}
The fifth Painlev\'e equation 
\begin{equation*}
\tag*{(P$_\mathrm{V}$)}
  y''=  \Bigl(\frac 1{2y} + \frac 1{y-1}
  \Bigr) (y')^{2} - \frac {y'}{ x}
  +\frac{(y-1)^2}{x^2} \Bigl(a_{\theta}y 
 - \frac{b_{\theta}} {y}\Bigr) + c_{\theta} \frac{y}{x} -\frac{y(y+1)}{2(y-1)},
 \end{equation*}
in which $8a_{\theta}=(\theta_0-\theta_1+\theta_{\infty} )^2,$ 
$8b_{\theta}=(\theta_0 -\theta_1 - \theta_{\infty})^2$, 
$c_{\theta}=1-\theta_0-\theta_1$ with $\theta_0, \theta_1, \theta_{\infty} 
\in\mathbb{C}$, 
governs the isomonodromy deformation of a linear system of the form
\begin{align*}
\frac{d\Xi}{d\lambda}= & \Bigl(\frac{x}2 \sigma_3+\frac{\mathcal{A}_0}{\lambda}
+\frac{\mathcal{A}_1}{\lambda-1} \Bigr)\Xi,
\\
&\sigma_3=\begin{pmatrix} 1 & 0 \\ 0 & -1  \end{pmatrix}, \quad
 \mathcal{A}_0=\begin{pmatrix}  \z+\theta_0/2  & -u(\z+\theta_0) \\
\z/u & -\z-\theta_0/2  
\end{pmatrix},  
\\
& \mathcal{A}_1=\begin{pmatrix} 
-\z-(\theta_0+\theta_{\infty})/2 & uy(\z+(\theta_0-\theta_1+\theta_{\infty})/2) 
\\
- (uy)^{-1}(\z+(\theta_0+\theta_1+\theta_{\infty})/2) 
 & \z+(\theta_0 +\theta_{\infty})/2  
\end{pmatrix}  
\end{align*}
(cf. \cite{A-K}, \cite[(1.1)]{S}, \cite[(3.1)]{S1}) with the
monodromy data $(M^0, M^1)=((m^0_{ij}), (m^1_{ij})) \in SL_2(\mathbb{C})^2$ 
defined along loops surrounding $\lambda=0$ and $ 1$, respectively.
Then a general solution $y(x)$ of (P$_{\mathrm{V}}$) is parametrised by 
$(M^0,M^1)$ \cite[Section 2]{A-K}. 
As in \cite[Theorem 2.1]{S, S1}, for each $\phi$ such that $0<|\phi|<\pi/2,$
$y(x)$ admits an expression of the form
\begin{equation}\label{1.1}
\frac{y(x)+1}{y(x)-1}= A_{\phi}^{1/2} \sn ((x-x_0)/2+\Delta(x); A^{1/2}_{\phi})
\end{equation}
with $\Delta(x)=O(x^{-2/9+\varepsilon})$ for any $\varepsilon$ satisfying
$0<\varepsilon<2/9$ as $x=e^{i\phi}t \to \infty$ through the cheese-like
strip $S(\phi,t_{\infty}, \kappa_0,\delta_0),$ where $v=\sn (z;k)$ is the
Jacobi elliptic function such that $v_z^2=(1-v^2)(1-k^2v^2),$ and
the symbols $A_{\phi}$, $x_0$ and $S(\phi,t_{\infty},\kappa_0,\delta_0)$ are
as in {\bf (1)} and {\bf (3)} below. 
Since $A_{\phi}$ does not depend on the solution $y(x)$, the leading term of 
the expression above contains the integration constant $x_0$ depending on 
$(M^0,M^1)$ and the other
integration constant appears in the error term $\Delta(x).$ 
Moreover $\Delta(x)$ may be treated in studying, say, the $\tau$-function
\cite[p.~121]{K-3}, and degeneration into trigonometric asymptotics
\cite[Section 4]{K-3}. 
For these facts detailed study on $\Delta(x)$ is desirable. 
Under the supposition $\Delta(x)=O(x^{-1})$, an asymptotic form 
of $\Delta(x)$ containing the other integration constant is discussed in 
\cite[Theorem 2.3 and Corollary 2.4]{S}.  
For the $\tau$-function associated with (P$_{\mathrm{I}}$) Iwaki \cite{Iwaki}, 
by the method of topological recursion, obtained a conjectural full-order 
expansion yielding the elliptic expression of solutions. 
\par
In this paper we unconditionally present an explicit expression of 
$\Delta(x)$, which leads to the estimate $\Delta(x)=O(x^{-1})$
as was conjectured.
The correction function $B_{\phi}(t)$ \cite[(5.5)]{S1} for
the Lagrangian of $y(x)$ contains information about asymptotics 
(see also \cite[Section 3]{K-3}). 
An analogous explicit formula is obtained for the error term of the asymptotic 
expression of $B_{\phi}(t)$. 
\par
Our results are stated in Theorems \ref{thm2.1}, \ref{thm2.2} and \ref{thm2.3}.  
In Section \ref{sc3}, from a system of equations equivalent to (P$_{\mathrm{V}}$)
we derive integral equations containing the error term $h(x)=
\Delta(x)/2.$ The final section is devoted to the proofs of main theorems 
by using these equations, in which 
our argument is quite different from those in \cite{I-K}, 
\cite[Chapter 8]{FIKN}, \cite{J-K} and \cite{Novo} applied to
(P$_{\mathrm{II}}$) and (P$_{\mathrm{I}}$).
\par
Throughout this paper we use the following symbols.
\par
{\bf (1)} For each $\phi\in \mathbb{R}$, $A_{\phi}\in \mathbb{C}$ is a unique 
solution of the Boutroux equations
\begin{equation*}
\re e^{i\phi} \int_{\mathbf{a}} \sqrt{\frac{A_{\phi}-z^2}{1-z^2} } dz =
\re e^{i\phi} \int_{\mathbf{b}} \sqrt{\frac{A_{\phi}-z^2}{1-z^2} } dz =0
\end{equation*}
\cite[Section 7]{S}. Here $\mathbf{a}$ and $\mathbf{b}$ are basic cycles as
in Figure \ref{cycles1} on the elliptic curve $\Pi^*=\Pi^*_+\cup \Pi^*_-$ 
given by $w(A_{\phi},z)=\sqrt{(1-z^2)(A_{\phi}-z^2)}$ 
such that $\Pi^*_+$ and $\Pi^*_-$ are
glued along the cuts $[-1,-A^{1/2}_{\phi}] \cup [A^{1/2}_{\phi},1]$ with
$0\le \re A_{\phi}^{1/2} \le 1;$ and the branches of the square roots
$$
\sqrt{\frac{A_{\phi}-z^2}{1-z^2}} = \frac{\sqrt{A_{\phi}-z^2}}{\sqrt{1-z^2}},
\quad
\sqrt{(A_{\phi}-z^2)(1-z^2)} = \sqrt{A_{\phi}-z^2}\, \sqrt{1-z^2}
$$
are determined by $z^{-1}\sqrt{A_{\phi}-z^2} \to i$ and 
$z^{-1}\sqrt{1-z^2} \to i$
as $z\to \infty$ on the upper sheet $\Pi^*_+$.
{\small
\begin{figure}[htb]
\begin{center}
\unitlength=0.8mm
\begin{picture}(80,37)(0,3)
\put(0,17){\makebox{$-1$}}
\put(22,3){\makebox{$-A_{\phi}^{1/2}$}}
\put(73,25){\makebox{$1$}}
\put(52,39){\makebox{$A_{\phi}^{1/2}$}}
\thinlines
\put(10,25.5){\line(2,-1){17}}
\put(10,24.5){\line(2,-1){17}}
\put(70,25.5){\line(-2,1){17}}
\put(70,24.5){\line(-2,1){17}}
\qbezier(35,33.5) (40,37) (46,38)
\qbezier(9,31.5) (14,32.7) (19,30.5)
\put(46,38){\vector(4,1){0}}
\put(9,31.5){\vector(-4,-1){0}}
\put(31,31){\makebox{$\mathbf{a}$}}
\put(20,29){\makebox{$\mathbf{b}$}}
\put(65,8){\makebox{$\Pi^*_{+}$}}
\thicklines
\put(10,25){\circle*{1}}
\put(27,16.5){\circle*{1}}
\put(53,33.5){\circle*{1}}
\put(70,25){\circle*{1}}
\qbezier(24,18) (25,25) (40,32.5)
\qbezier(40,32.5) (56,40) (56.5,31.7)
\qbezier[15](40,18) (54,25.5) (56,30)
\qbezier[15](24,16) (26,11) (40,18)
\qbezier(6,27) (10,32.5) (25.9,24.8)
\qbezier(29,23.2) (35.5,18) (34,13)
\qbezier(6,27) (3.5,21) (14,15)
\qbezier (14,15)(30,7)(34,13)
\end{picture}
\end{center}
\caption{Cycles $\mathbf{a},$ $\mathbf{b}$ on $\Pi^*$}
\label{cycles1}
\end{figure}
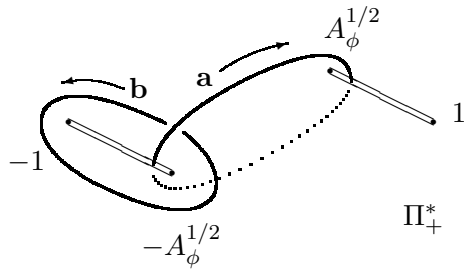
}
\par
{\bf (2)} The periods of $\Pi^*$ along $\mathbf{a}$ and $\mathbf{b}$ are
$$
\Omega_{\mathbf{a}} =\int_{\mathbf{a}} \frac{dz}{w(A_{\phi}, z)}, \quad
\Omega_{\mathbf{b}} =\int_{\mathbf{b}} \frac{dz}{w(A_{\phi}, z)}, 
$$
and write
$$
\mathcal{E}_{\mathbf{a}} =\int_{\mathbf{a}} \sqrt{\frac{A_{\phi}-z^2}
{1-z^2}}\, dz, \quad
\mathcal{E}_{\mathbf{b}} =\int_{\mathbf{b}} \sqrt{\frac{A_{\phi}-z^2}
{1-z^2}}\, dz. 
$$
\par
{\bf (3)} Set
$$
x_0 \equiv \frac{-1}{\pi i} (\Omega_{\mathbf{b}} \log(m^0_{21}m^1_{12})
+\Omega_{\mathbf{a}} \log \mathfrak{m}_{\phi} ) -(\tfrac 12 \Omega_{\mathbf{a}}
+\Omega_{\mathbf{b}})(\theta_{\infty}+1) -\tfrac 12 \Omega_{\mathbf{a}}
 \mod 2\Omega_{\mathbf{a}}
\mathbb{Z} +2\Omega_{\mathbf{b}}\mathbb{Z},
$$
in which $\mathfrak{m}_{\phi}=m^0_{11}$ if $-\pi/2<\phi<0$, and $=e^{-\pi i
\theta_{\infty}} (m^1_{11})^{-1}$ if $0<\phi<\pi/2.$ For given positive
numbers $\kappa_0,$ $\delta_0$ and $t_{\infty}$,
$$
S(\phi,t_{\infty}, \kappa_0,\delta_0)=\{ x=e^{i\phi}t \,|\, \re t>t_{\infty},
\,\,\, |\im t|<\kappa_0 \}\setminus \bigcup_{\sigma\in \mathcal{P}_0}
\{|x-\sigma|<\delta_0 \}
$$
with 
$
\mathcal{P}_0 =\{\sigma \,|\, \sn((\sigma-x_0)/2; A^{1/2}_{\phi})=\infty\}
=\{x_0+\Omega_{\mathbf{a}}\mathbb{Z} +\Omega_{\mathbf{b}}(2\mathbb{Z}+1)\},
$
and
$$
\check{S}(\phi,t_{\infty},\kappa_0,\delta_0)=
{S}(\phi,t_{\infty},\kappa_0,\delta_0)\setminus 
 \bigcup_{\sigma\in \mathcal{Q}} \{|x-\sigma|<\delta_0 \}
$$
with
$\mathcal{Q}= \{{\sigma} \,|\, \sn (({\sigma}-x_0)/2; A^{1/2}
_{\phi})= \pm A^{-1/2}_{\phi}, \pm 1 \},$ in which $\delta_0$ is also supposed
so small that $\{|x-\sigma_1|=\delta_0\} \cap \{|x-\sigma_2|=\delta_0\}=
\emptyset$ for any $\sigma_1,$ $\sigma_2 \in \mathcal{P}_0\cup \mathcal{Q},$
$\sigma_1\not=\sigma_2.$
For $\sigma=e^{i\phi}t_{\sigma} \in \mathcal{Q}$ let $l(\sigma)$ be the line 
defined
by $x=e^{i\phi}(\re t_{\sigma}+i\eta)$ with $\eta \ge \im t_{\sigma}$ 
if $\im t_{\sigma} \ge 0$ (respectively, $\eta \le \im t_{\sigma}$ if 
$\im t_{\sigma} <0$); and, if necessary, modify $l(\sigma)$ 
not to touch other circles $\{|x-\sigma'|=\delta_0\}$ with $\sigma'\in
\mathcal{P}_0\cup\mathcal{Q}\setminus\{\sigma\}$
by suitable replacement of local segments on $l(\sigma)$ with arcs. Then let 
$\check{S}_{\mathrm{cut}}(\phi,t_{\infty},\kappa_0, \delta_0)$ denote
$\check{S}(\phi,t_{\infty},\kappa_0, \delta_0)$ equipped with the cuts along  
$l(\sigma)$ or its modification for all $\sigma \in \mathcal{Q}$ (cf.
Figure \ref{strips}).   
{\small
\begin{figure}[htb]
\begin{center}
\unitlength=0.85mm
\begin{picture}(80,70)(-40,-45)

  \qbezier (-27,11) (0,20) (27,29)
  \qbezier (-27,-29) (0,-20) (27,-11)

\put(20,10){\circle{1.4}} \put(20,10){\circle{3}}
\put(20,-10){\circle{1.4}} \put(20,-10){\circle{3}}
\put(-20,10){\circle{1.4}} \put(-20,10){\circle{3}}
\put(-20,-10){\circle{1.4}} \put(-20,-10){\circle{3}}
\put(0,10){\circle{1.4}} \put(0,10){\circle{3}}
\put(0,-10){\circle{1.4}} \put(0,-10){\circle{3}}

\put(10,0){\circle*{0.7}} \put(10,0){\circle{3}}
\put(10,10){\circle*{0.7}} \put(10,10){\circle{3}}
\put(10,20){\circle*{0.7}} \put(10,20){\circle{3}}
\put(10,-10){\circle*{0.7}} \put(10,-10){\circle{3}}

\put(-10,0){\circle*{0.7}} \put(-10,0){\circle{3}}
\put(-10,10){\circle*{0.7}} \put(-10,10){\circle{3}}
\put(-10,-20){\circle*{0.7}} \put(-10,-20){\circle{3}}
\put(-10,-10){\circle*{0.7}} \put(-10,-10){\circle{3}}
\put(-10,-20){\circle*{0.7}} \put(-10,-20){\circle{3}}

\put(23.5,-31){\circle{1.4}} \put(26,-32){\makebox{$\in \mathcal{P}_0$\,,}}
\put(42.8,-31){\circle*{0.7}} \put(45,-32){\makebox{$\in \mathcal{Q}$}}

\put(-20,-43){\makebox{(a) \quad$\check{S}(\phi,t_{\infty},\kappa_0,\delta_0)$}}
\end{picture}
\qquad\quad
\begin{picture}(80,70)(-40,-45)

  \qbezier (-27,11) (0,20) (27,29)
  \qbezier (-27,-29) (0,-20) (27,-11)

\put(20,10){\circle{1.4}} \put(20,10){\circle{3}}
\put(20,-10){\circle{1.4}} \put(20,-10){\circle{3}}
\put(-20,10){\circle{1.4}} \put(-20,10){\circle{3}}
\put(-20,-10){\circle{1.4}} \put(-20,-10){\circle{3}}
\put(0,10){\circle{1.4}} \put(0,10){\circle{3}}
\put(0,-10){\circle{1.4}} \put(0,-10){\circle{3}}

\put(10,0){\circle*{0.7}} \put(10,0){\circle{3}}
\put(10,10){\circle*{0.7}} \put(10,10){\circle{3}}
\put(10,20){\circle*{0.7}} \put(10,20){\circle{3}}
\put(10,-10){\circle*{0.7}} \put(10,-10){\circle{3}}

\qbezier (9.53,21.42) (9.03,22.92)  (8.97,23.10)
\qbezier (-9.53,-21.42) (-9.03,-22.92)  (-8.97,-23.10)

\put(9.53,11.42){\line(-1,3){3.5}} 
\put(-9.53,-11.42){\line(1,-3){3.5}} 
\put(10.47,-1.42){\line(1,-3){4.5}} 
\put(-10.47,1.42){\line(-1,3){4.5}} 
\put(10.47,-11.42){\line(1,-3){1.5}} 
\put(-10.47,11.42){\line(-1,3){1.5}} 

\put(-10,0){\circle*{0.7}} \put(-10,0){\circle{3}}
\put(-10,10){\circle*{0.7}} \put(-10,10){\circle{3}}
\put(-10,-20){\circle*{0.7}} \put(-10,-20){\circle{3}}
\put(-10,-10){\circle*{0.7}} \put(-10,-10){\circle{3}}

\put(-20,-43){\makebox{(b) \quad$\check{S}_{\mathrm{cut}}
(\phi,t_{\infty},\kappa_0,\delta_0)$}}
\end{picture}
\end{center}
\caption{Cheese-like strips} 
\label{strips}
\end{figure}
}
\par 
{\bf (4)} For $\im \tau >0,$
$$
\vartheta(z,\tau)=\sum_{n\in \mathbb{Z}} e^{\pi i\tau n^2 +2\pi i zn}
$$
with $\vartheta'(z,\tau) =(d/dz)\vartheta(z,\tau)$ is the $\vartheta$-function
\cite{H}, \cite{WW}. Note that $\vartheta(z\pm 1,\tau)=\vartheta(z,\tau),$
$\vartheta(z\pm \tau,\tau)=e^{-\pi i(\tau \pm 2z)}\vartheta(z,\tau).$
\par
{\bf (5)} We write $f \ll g$ or $g \gg f$ if $f=O(g).$
\section{Main results}\label{sc2}
Our results are stated as follows.
\begin{thm}\label{thm2.1}
Suppose that 
$0<|\phi|< \pi/2.$ Let $y(x)$ be the solution of $(\mathrm{P}_{\mathrm{V}})$ 
corresponding to $(M^0,M^1)=((m_{ij}^0),(m_{ij}^1))$ 
with $m_{11}^0 m_{11}^1 m_{21}^0m_{12}^1 \not=0.$ Then 
$$
\frac{y(x) +1}{y(x)-1}
=A^{1/2}_{\phi} \mathrm{sn} ((x-x_0)/2 ; A^{1/2}_{\phi}) +O(x^{-1})
$$
as $x \to \infty$ through the cheese-like strip
$ S(\phi, t_{\infty}, \kappa_0, \delta_0)$, where
$\kappa_0 $ is a given number, $\delta_0$ a given small number,
and $t_{\infty}=t_{\infty}(\kappa_0,\delta_0)$ a large number depending on 
$(\kappa_0,\delta_0).$ 
\end{thm}
Set
\begin{align*}
\psi_0(x)&= A_{\phi}^{1/2} \sn ((x-x_0)/2; A_{\phi}^{1/2}),
\\
b_0(x)&=\beta_0 -\frac{2\mathcal{E}_{\mathbf{a}}}{\Omega_{\mathbf{a}}}x
-\frac{8}{\Omega_{\mathbf{a}}} \frac{\vartheta'}{\vartheta}\Bigl(\frac 1
{2\Omega_{\mathbf{a}}}(x-x_0), \tau_0 \Bigr), \quad 
\tau_0=\frac{\Omega_{\mathbf{b}}}{\Omega_{\mathbf{a}}},
\\
\mathfrak{b}(x)&=\frac{\mathcal{E}_{\mathbf{a}}}4(x-x_0)+
 \frac{\vartheta'}{\vartheta}\Bigl(\frac 1
{2\Omega_{\mathbf{a}}}(x-x_0), \tau_0 \Bigr) 
 =-\frac{\Omega_{\mathbf{a}}}{8}(b_0(x)-b_0(x_0)) 
\end{align*}
with
$$
\beta_0 = -\frac 8{\Omega_{\mathbf{a}}}( \log (m^0_{21}m^1_{12}) 
+\pi i(\theta_{\infty}+1)), \quad
 b_0(x_0)=\beta_0-\frac{2\mathcal{E}_{\mathbf{a}}}{\Omega_{\mathbf{a}}}x_0,
$$
where $\psi_0(x),$ $b_0(x)$ and $\mathfrak{b}(x)$ are bounded in 
$S(\phi,t_{\infty},\kappa_0,\delta_0)$. 
\begin{thm}\label{thm2.2}
The error term $\Delta(x)=h(x)/2$ in \eqref{1.1} is represented by
\begin{equation*}
h(x)= -\frac{2((\theta_0-\theta_1)^2+\theta_{\infty}^2)}{A_{\phi}-1}x^{-1}
 - \int^x_{\infty}F_1(\psi_0, b_0)\frac{d\xi}{\xi} -\frac 32 
\int^x_{\infty} F_1(\psi_0,b_0)^2 \frac{d\xi}{\xi^2} +O(x^{-2}),
\end{equation*}
with
$$
F_1(\psi_0,b_0)= \frac{4(\theta_0+\theta_1)\psi_0 -b_0}{2(A_{\phi}-\psi_0^2)},
\quad \psi_0=\psi_0(\xi), \quad b_0=b_0(\xi).
$$
Here 
$$
\int^x_{\infty} F_1(\psi_0,b_0) \frac{d\xi}{\xi} \ll x^{-1}, \quad
\int^x_{\infty} F_1(\psi_0,b_0)^2 \frac{d\xi}{\xi^2} \ll x^{-1}
$$
as $x \to \infty$ through $\check{S}_{\mathrm{cut}}
(\phi,t_{\infty},\kappa_0,\delta_0).$
Furthermore, 
$$
x h(x)= h_0\beta_0^2 +h_1(x) \beta_0 +h_2(x) +O(x^{-1}),
$$
where $h_0= (1/8)A_{\phi}^{-1}(1-A_{\phi})^{-1}$, $h_1(x)\ll 1,$ $h_2(x)\ll 1.$
\end{thm}
\begin{rem}\label{rem2.0}
Since $(y(x)+1)(y(x)-1)^{-1}=\psi_0(x+h(x))$,
we have, in each neighbourhood of $\sigma\in \check{S}_{\mathrm{cut}}
(\phi,t_{\infty},\kappa_0,\delta_0)$,
\begin{equation*}
\frac{y(x)+1}{y(x)-1} -A^{1/2}_{\phi} \mathrm{sn}((x-x_0)/2;A_{\phi}^{1/2})
=\sum_{j=1}^{\infty} \frac {\psi_0^{(j)}(x)}{j!} h(x)^j \sim \psi_0'(x)h(x),
\end{equation*}
which implies the single-valuedness of $h(x)$ in $\check{S}(\phi,
t_{\infty},\kappa_0,\delta_0).$  
In showing Theorem \ref{thm2.2}, for convenience' sake, the integral 
representations have been treated
in $\check{S}_{\mathrm{cut}}(\phi,t_{\infty},\kappa_0,\delta_0)$, in which
a contour joining $x$ to $\infty$ is topologically specified, to avoid the
possible multi-valuedness of each integral around the pole 
at $\sigma \in \mathcal{Q}.$
\end{rem}
\begin{rem}\label{rem2.0a}
Note that (P$_{\mathrm{V}}$) is equivalent to the system
\begin{align*}
&\z= \frac {x(y-y')}{2(y-1)^2}+\frac{\theta_0+\theta_1}{2(y-1)}-\frac 14
(\theta_0-\theta_1+\theta_{\infty}),
\\
&x\z'= y\z(\z+\tfrac 12(\theta_0-\theta_1+\theta_{\infty}))
-y^{-1}(\z+\theta_0)(\z+\tfrac 12(\theta_0+\theta_1+\theta_{\infty}))
\end{align*}
governing the isomonodromy deformation \cite{A-K}. 
By Theorem \ref{thm2.2}, from \eqref{1.1} it follows that
$$
-\frac{y'(x)}{(y(x)-1)^2}=\frac{A_{\phi}^{1/2}}8 
\Bigl(2\,\mathrm{sn}'(\tfrac 12 (x-x_0))(1+h'(x))+
\mathrm{sn}''(\tfrac 12(x-x_0))h(x)\Bigr) +O(x^{-2})
$$
and $2(y-1)^{-1}=A^{1/2}_{\phi}(\mathrm{sn}(\tfrac 12(x-x_0))
+\tfrac 12\mathrm{sn}'(\tfrac 12(x-x_0))h(x))-1$
in $S(\phi,t_{\infty},\kappa_0,\delta_0)$,
where $h(x)$, $h'(x) \ll x^{-1}$, 
$\mathrm{sn}\,z=\mathrm{sn}(z,A^{1/2}_{\phi})$ 
and $\mathrm{sn}'z=\frac d{dz}\mathrm{sn}\,z.$ Then we have
\begin{align*}
\z(x) =& \frac x8\bigl(A^{1/2}_{\phi}\mathrm{sn}'(\tfrac 12(x-x_0))+
A_{\phi}\mathrm{sn}^2(\tfrac 12(x-x_0))-1\bigr)
\\
&+ \frac {xh(x)}{16}\bigl(A^{1/2}_{\phi}\mathrm{sn}''(\tfrac 12(x-x_0))+
2A_{\phi}\mathrm{sn}(\tfrac 12(x-x_0))\mathrm{sn}'(\tfrac 12(x-x_0))\bigr)
\\
&+ \frac {xh'(x)}{8}A^{1/2}_{\phi}\mathrm{sn}'(\tfrac 12(x-x_0))
+ \frac {\theta_0+\theta_1}{4}A^{1/2}_{\phi}\mathrm{sn}(\tfrac 12(x-x_0))
-\frac {2\theta_0+\theta_{\infty}}4+O(x^{-1}).
\end{align*}
\end{rem}
Recall the correction function $B_{\phi}(t)$ such that $a_{\phi}=A_{\phi}+
t^{-1}B_{\phi}(t),$ where
\begin{align*}
a_{\phi}=& 1-\frac{4(e^{-2i \phi}(y^*)^2-y^2)}{y(y-1)^2} +4e^{-i\phi}
(\theta_0+\theta_1)\frac{y+1}{y-1} t^{-1}
\\
& + e^{-2i\phi} \frac{(y-1)}{y} ((\theta_0-\theta_1+\theta_{\infty})^2y
-(\theta_0-\theta_1-\theta_{\infty})^2) t^{-2}
\end{align*} 
with $x=e^{i\phi}t$ \cite[(3.5)]{S1}. In particular, $a_{\phi, \mathrm{Lag}}:
=a_{\phi}|_{y^*=dy/dt}$ is the Lagrangian of $y=y(e^{i\phi}t).$ In this case,
let $b(x)$ be such that
$$
a_{\phi,\mathrm{Lag}}=A_{\phi}+\frac{b(x)}x = A_{\phi}+\frac{e^{-i\phi}
b(e^{i\phi}t)}{t}.
$$
If $y^*=dy/dt$, then $b(x)=e^{i\phi}B_{\phi}(t).$ (In \cite{S}, $a_{\phi}$ and
$B_{\phi}(t)$ are defined under the condition $y^*=dy/dt.$)
\begin{thm}\label{thm2.3}
Under the same suppositions as in Theorems $\ref{thm2.1}$ and $\ref{thm2.2}$,
\begin{align*}
b(x)-b_0(x)= b'_0(x)h(x) &
- 4((\theta_0-\theta_1)^2 +\theta_{\infty}^2)x^{-1}
\\
& -\int^x_{\infty} 
(A_{\phi}-\psi_0^2)F_1(\psi_0,b_0)^2\frac{d\xi}{\xi^2} +O(x^{-2}),
\end{align*}
in which $b'_0(x)=4\psi_0'-2(A_{\phi}-\psi_0^2)$, and
$$
\int^x_{\infty}(A_{\phi}-\psi_0^2)F_1(\psi_0,b_0)^2\frac{d\xi}{\xi^2} \ll x^{-1}
$$
as $x\to \infty$ through $\check{S}_{\mathrm{cut}}
(\phi,t_{\infty},\kappa_0,\delta_0)$, and $b(x)-b_0(x) \ll x^{-1}$ in 
$S(\phi, t_{\infty},\kappa_0,\delta_0).$
\end{thm}
\begin{rem}\label{rem2.1}
As calculated in Section \ref{sc4} the integrals above are written in the
form
\begin{align*}
\int^x_{\infty} F_1(\psi_0,& b_0)  \frac{d\xi}{\xi} =
2(\theta_0+\theta_1)\int^x_{\infty}\frac{\psi_0}{A_{\phi}-\psi_0^2} \frac{d\xi}
{\xi}
\\
& -\frac {b_0(x_0)}2 \int^x_{\infty} \frac 1{A_{\phi}-\psi_0^2} \frac{d\xi}
{\xi} 
 +\frac 4{\Omega_{\mathbf{a}}} \int^x_{\infty} \frac{\mathfrak{b}}
{A_{\phi}-\psi_0^2}\frac{d\xi}{\xi},
\\
\int^x_{\infty} F_1(\psi_0,& b_0)^2 \frac{d\xi}{\xi^2} =
\frac{16(\theta_0+\theta_1)^2A_{\phi} +b_0(x_0)^2}{12A_{\phi}
(A_{\phi}-1)}x^{-1}+\frac{4b_0(x_0)}{3A_{\phi}(A_{\phi}-1)\Omega_{\mathbf{a}}}
\int^x_{\infty}\mathfrak{b}\frac{d\xi}{\xi^2}
\\
& +\frac{16(\theta_0+\theta_1)}{\Omega_{\mathbf{a}}}\int^x_{\infty}\frac{
\mathfrak{b}\psi_0}{(A_{\phi}-\psi_0^2)^2}\frac{d\xi}{\xi^2}
+\frac{16}{\Omega_{\mathbf{a}}^2} \int^x_{\infty} \frac{\mathfrak{b}^2}
{(A_{\phi}-\psi_0^2)^2}\frac{d\xi}{\xi^2} +O(x^{-2}),
\\
\int^x_{\infty} (A_{\phi}- & \psi_0^2)  F_1(\psi_0,b_0)^2\frac{d\xi}{\xi^2} =
4(\theta_0+\theta_1)^2 x^{-1} 
\\
& +\frac{16(\theta_0+\theta_1)}{\Omega_{\mathbf{a}}}\int^x_{\infty}\frac{
\mathfrak{b}\psi_0}{A_{\phi}-\psi_0^2}\frac{d\xi}{\xi^2}
+\frac{16}{\Omega_{\mathbf{a}}^2} \int^x_{\infty} \frac{\mathfrak{b}^2}
{A_{\phi}-\psi_0^2}\frac{d\xi}{\xi^2} +O(x^{-2}),
\end{align*}
in which each integral on the right-hand 
sides is $O(x^{-1})$, and immediately yield detailed expressions of $h(x)$ 
and $b(x)-b_0(x)$ (see Section \ref{ssc4.3}). 
\end{rem}
\section{System of integral equations}\label{sc3}
To prove our theorems, we recall the following facts \cite[Section 6]{S}.
\par
(1) For the solution $y(x)$ of (P$_{\mathrm{V}}$), $(\psi(x),b(x))$ with
$\psi(x)=(y(x)+1)(y(x)-1)^{-1}$ solves a system of equations
\begin{align}\notag
4(\psi')^2 =& (1-\psi^2)(A_{\phi}-\psi^2) -(1-\psi^2)(4(\theta_0+\theta_1)\psi
-b)x^{-1}
\\
\label{3.1}
& +4(2(\theta_0-\theta_1)\theta_{\infty} \psi +(\theta_0-\theta_1)^2+
\theta_{\infty}^2) x^{-2},
\\
\label{3.2}
b'=& -2(A_{\phi}-\psi^2)+ 4\psi' + (4(\theta_0+\theta_1)\psi -b)x^{-1},
\end{align}
where $b=b(x)$ is as defined in Section \ref{sc2} by using the Lagrangian
$a_{\phi, \mathrm{Lag}}.$
\par
(2) $\psi_0(x)$ and $b_0(x)$ are bounded in $S(\phi,t_{\infty},\kappa_0,\delta
_0)$ and fulfil
\begin{align}\label{3.3}
4(\psi'_0)^2 &= (1-\psi_0^2)(A_{\phi}-\psi_0^2),
\\
\label{3.4}
b'_0 &= -2(A_{\phi}-\psi_0^2) +4\psi_0'
\end{align}
\cite{S}, which at least formally approximates system \eqref{3.1}, \eqref{3.2}.
\begin{prop}\label{prop3.1}
Equations \eqref{3.1}, \eqref{3.2} admit a solution $(\psi(x), b(x))$ such that
$\psi(x)=\psi_0(x+h(x))$ with $h(x) \ll x^{-2/9+\varepsilon}$ and 
$b(x)-b_0(x) \ll x^{-2/9+\varepsilon}$ as $x \to \infty$ through $S(\phi,
t_{\infty},\kappa_0,\delta_0)$ for any $0<\varepsilon < 2/9.$
Furthermore $b_0(x)$ and $b(x)$ are bounded.
\end{prop}
\begin{proof}
As in \cite[Section 5]{S1}, the correction function $b^*(x):
=e^{i\phi}B_{\phi}(t)$ admits the asymptotic expression $b^*(x)-b_0(x) \ll
x^{-2/9+\varepsilon},$ which follows from \cite[(5.5)]{S1} with $\delta
=2/9-\varepsilon.$ (Here we note that, in the argument of \cite[Sectios 4 and 5]
{S1} as well, $\delta$ is so chosen in accordance with the annulus $\mathcal{A}
_{\varepsilon}$ in \cite[p.~64]{S}.) By the justification scheme \cite[Section 5]
{S1} with \cite{K}, for $b(x)$ corresponding to the Lagrangian $a_{\phi,
\mathrm{Lag}}$ as well, the estimate $b(x)-b_0(x)\ll x^{-2/9+\varepsilon}$
remains valid. 
\end{proof}
From Proposition \ref{prop3.1} with \eqref{3.3}, it follows that
$
2\psi'(x)=2(1+h'(x)) \psi'_0(x+h(x)) =(1+h')\sqrt{(1-\psi_0(x+h)^2)
(A_{\phi}-\psi_0(x+h)^2)} =(1+h')\sqrt{(1-\psi^2)(A_{\phi}-\psi^2)}.
$
Then \eqref{3.1} becomes $(1+h')^2=1-2F_1(\psi,b)x^{-1}+2F_2(\psi)x^{-2}$,
which yields
\begin{equation}\label{3.41}
h'=-F_1(\psi, b)x^{-1} +(F_2(\psi)-\tfrac 12 F_1(\psi,b)^2)x^{-2} +O(x^{-3})
\end{equation}
in $\check{S}(\phi,t_{\infty},\kappa_0,\delta_0),$ where
$$
F_1(\psi,b)=\frac{4(\theta_0+\theta_1)\psi-b}{2(A_{\phi}-\psi^2)}, \quad
F_2(\psi)=\frac{2(2(\theta_0-\theta_1)\theta_{\infty}\psi +(\theta_0-\theta_1)^2
+ \theta_{\infty}^2)}{(1-\psi^2)(A_{\phi}-\psi^2)}.
$$
Using $\psi=\psi_0+\psi'_0h+O(h^2),$ we have
\begin{align}\notag
h'=& -F_1(\psi_0,b)x^{-1} +(F_2(\psi_0)-\tfrac 12 F_1(\psi_0,b)^2)x^{-2}
\\
\label{3.5}
& -(F_1)_{\psi}(\psi_0,b)\psi'_0 hx^{-1} +O(x^{-1}(|x^{-1}|+|h|)^2). 
\end{align}
In what follows we suppose that, for a positive number $\mu \le 1$,
\begin{equation}\label{3.6}
h(x) \ll x^{-\mu}  \phantom{---------}
\end{equation}
in $\check{S}_{\mathrm{cut}}(\phi,t_{\infty},\kappa_0,\delta_0)$. 
By Proposition \ref{prop3.1},
estimate \eqref{3.6} is true if, say, $\mu=1/9.$ 
\par
Let $\{x_{\nu}\} \subset \check{S}_{\mathrm{cut}}(\phi,t_{\infty}, 
\kappa_0,\delta_0)$ be
a given sequence such that $|x_1|< \cdots < |x_{\nu}| < \cdots,$ $|x_{\nu}|
\to \infty$. Then, by \eqref{3.2} and \eqref{3.4}
\begin{align*}
& b(x)-b(x_{\nu}) = \int^x_{x_{\nu}} (4\psi'-2(A_{\phi}-\psi^2)) d\xi
+\int^x_{x_{\nu}} (4(\theta_0+\theta_1)\psi -b)\frac{d\xi}{\xi},
\\
& b_0(x)-b_0(x_{\nu}) = \int^x_{x_{\nu}} (4\psi'_0-2(A_{\phi}-\psi_0^2)) d\xi,
\end{align*}
from which we derive, for $x\in \check{S}_{\mathrm{cut}}(\phi,t_{\infty},
\kappa_0, \delta_0)$ with $|x|<|x_{\nu}|$,
\begin{align}\notag
b(x)-b_0(x) - &(b(x_{\nu})-b_0(x_{\nu}))= 4(\psi(x)-\psi_0(x) -(\psi(x_{\nu})
-\psi_0(x_{\nu})))
\\
\label{3.7}
& +2\int^x_{x_{\nu}} (\psi^2-\psi_0^2) d\xi +2\int^x_{x_{\nu}} (A_{\phi}-\psi^2)
F_1(\psi,b)\frac{d\xi}{\xi}.
\end{align}
In this equality, by \eqref{3.6} and Proposition \ref{prop3.1},
\begin{align*}
&\psi(x)-\psi_0(x)-(\psi(x_{\nu})-\psi_0(x_{\nu})) \ll |h(x)|+|h(x_{\nu})|
\ll |x^{-\mu}|+|x_{\nu}^{-\mu}|,
\\
& b(x_{\nu})-b_0(x_{\nu}) \ll x_{\nu}^{-2/9+\varepsilon}.
\end{align*}
Furthermore,
\begin{align*}
2\int^x_{x_{\nu}} (\psi^2-\psi_0^2)d\xi &
= 2\int^x_{x_{\nu}} \Bigl((\psi_0^2)'h +\frac{(\psi_0^2)''}2 h^2+ \cdots
+\frac{(\psi_0^2)^{(p)}}{p!} h^p +O(h^{p+1}) \Bigr)d\xi
\\
& = -2\int^x_{x_{\nu}} \Bigl(\psi_0^2 +(\psi_0^2)' h+ \cdots
+\frac{(\psi_0^2)^{(p-1)}}{(p-1)!} h^{p-1} \Bigr)h'd\xi +O(|x^{-\mu}|+|x_{\nu}
^{-\mu}|),
\end{align*}
if $-(p+1)\mu+1 \le -\mu,$ i.e. $-p\mu+1\le 0$; and by \eqref{3.41} and 
\eqref{3.6},
\begin{align*}
2\int^x_{x_{\nu}} (A_{\phi}-\psi^2)F_1(\psi,b)\frac{d\xi}{\xi} 
=& -2\int^x_{x_{\nu}}(A_{\phi}-\psi^2) (h' +O(\xi^{-2}) )d\xi
\\
 =& -2\int^x_{x_{\nu}} (A_{\phi}-\psi^2) h' d\xi  + O(|x^{-1}|+|x_{\nu}^{-1}|),
\end{align*}
where
\begin{align*}
&-2 \int^x_{x_{\nu}} (A_{\phi}-\psi^2)h' d\xi 
\\
=& 2\int^x_{x_{\nu}} \Bigl(\psi_0^2 +(\psi_0^2)' h+ \cdots
+\frac{(\psi_0^2)^{(p-1)}}{(p-1)!} h^{p-1} \Bigr)h' d\xi
+ O(|x^{-\mu}|+|x_{\nu}^{-\mu}|),
\end{align*}
since $h' \ll \xi^{-1}$ by \eqref{3.5} and \eqref{3.6}. 
Insert these quantities with $p$ such
that $-p\mu+1 \le 0$ into \eqref{3.7}. Under the passage to the limit $x_{\nu}
\to \infty$, we arrive at the estimate
\begin{equation}\label{3.81}
b(x)-b_0(x) \ll x^{-\mu}
\end{equation}
in $\check{S}_{\mathrm{cut}}(\phi,t_{\infty},\kappa_0,\delta_0).$ 
Then equation \eqref{3.5} is written in the form
\begin{align}\notag
h'&= -F_1(\psi_0,b)x^{-1} +(F_2(\psi_0)-\tfrac 12 F_1(\psi_0,b_0)^2)x^{-2}
-(F_1)_{\psi}(\psi_0,b_0)\psi_0' hx^{-1} +O(x^{-1-2\mu})
\\
\label{3.82}
&= -F_1(\psi_0,b_0) x^{-1} +O(x^{-1-\mu}).
\end{align}
For any sequence $\{x_{\nu}\} \subset \check{S}_{\mathrm{cut}}
(\phi,t_{\infty},\kappa_0,\delta_0)$, integration of this yields
\begin{align*}
&h(x)-h(x_{\nu})
\\
=&-\int^x_{x_{\nu}} F_1(\psi_0,b)\frac{d\xi}{\xi} +\int^x_{x_{\nu}}
(F_2(\psi_0)-\tfrac 12 F_1(\psi_0,b_0)^2)\frac{d\xi}{\xi^2} -\mathcal{I}_0
+O(|x^{-2\mu}|+|x_{\nu}^{-2\mu}|)
\end{align*}
with
\begin{align*}
\mathcal{I}_0 &= \int^x_{x_{\nu}} (F_1)_{\psi}(\psi_0,b_0)\psi_0' h \frac{d\xi}
{\xi}
\\
&= \int^x_{x_{\nu}} \Bigl(F_1(\psi_0,0)_{\xi} -  \Bigl(\frac 1{2(A_{\phi}
-\psi_0^2)} \Bigr)_{\! \xi} b_0 \Bigr) h\frac{d\xi}{\xi}
\\
&= \int^x_{x_{\nu}}  \Bigl(F_1(\psi_0,0)F_1(\psi_0,b_0) -\frac{b_0 F_1(\psi_0,
b_0) -b_0'h\xi}{2(A_{\phi}-\psi_0^2)}  \Bigr) \frac{d\xi}{\xi^2}
+O(|x^{-1-\mu}|+|x_{\nu}^{-1-\mu}|)
\\
&= \int^x_{x_{\nu}}  F_1(\psi_0,b_0)^2 \frac{d\xi}{\xi^2} 
+\frac 12 \int^x_{x_{\nu}} \frac{b_0'h}{A_{\phi}-\psi_0^2}  \frac{d\xi}{\xi}
+ O(|x^{-1-\mu}|+|x_{\nu}^{-1-\mu}|),
\end{align*}
where the third line is due to integration by parts.
Hence we have, in $\check{S}_{\mathrm{cut}}(\phi,t_{\infty},\kappa_0,\delta_0)$,
\begin{align}\notag
h(x)=& -\int^x_{\infty}F_1(\psi_0,b)\frac{d\xi}{\xi}
\\
\label{3.9}
& +\int^x_{\infty}
\Bigl(F_2(\psi_0)-\frac 32 F_1(\psi_0,b_0)^2 \Bigr) \frac{d\xi}{\xi^2}
-\frac 12 \int^x_{\infty}\frac{b_0' h}{A_{\phi}-\psi_0^2}\frac{d\xi}{\xi}
+O(x^{-2\mu}),
\end{align}
in which the convergence of $\int^x_{\infty} F_1(\psi_0,b)\xi^{-1}d\xi$ is
guaranteed by the absolute convergence of the remaining two integrals.
\par
By \eqref{3.1} and \eqref{3.2},
$$
b'=-(A_{\phi}-\psi^2)+4\psi' -\frac{4(\psi')^2}{1-\psi^2}+2(A_{\phi}-\psi^2)
F_2(\psi)x^{-2}.
$$
From this combined with \eqref{3.4} and $2\psi'=(1+h')\sqrt{(1-\psi^2)(A_{\phi}
-\psi^2)},$ it follows that
$$
(b-b_0)'=2(\psi^2-\psi_0^2)+4(\psi-\psi_0)' -2h'(A_{\phi}-\psi^2)(1+h'/2)
+2(A_{\phi}-\psi^2)F_2(\psi)x^{-2},
$$
and then, for any $\{x_{\nu}\} \subset \check{S}_{\mathrm{cut}}
(\phi,t_{\infty},\kappa_0, \delta_0)$, $\chi:=b-b_0$ satisfies
\begin{align*}
\chi(x)-\chi(x_{\nu}) =& 4(\psi-\psi_0)-4(\psi(x_{\nu})-\psi_0(x_{\nu}))
 +2\int^x_{x_{\nu}} (\psi^2-\psi_0^2+h'\psi^2)d\xi
\\
 &-2A_{\phi}(h(x)-h(x_{\nu})) 
-\int^x_{x_{\nu}}(A_{\phi}-\psi^2)((h')^2-2F_2(\psi)\xi^{-2})d\xi.
\end{align*}
Observing that
\begin{align*}
& 2\int^x_{x_{\nu}}(\psi^2-\psi_0^2+h'\psi^2) d\xi
\\
=& 2\int^x_{x_{\nu}} \Bigl( (\psi_0^2)'h +\cdots + \frac{(\psi_0^2)^{(p)}}{p!}
h^p
\\
&\phantom{---} 
+ h' \Bigl( \psi_0^2 +\cdots + \frac{(\psi_0^2)^{(p-1)}}{(p-1)!}h^{p-1}
\Bigr) + O(|h^{p+1}|+|h^ph'|) \Bigr)d\xi
\\
=& 2\int^x_{x_{\nu}} \Bigl( 
 \Bigl( \psi_0^2 h +\cdots + \frac{(\psi_0^2)^{(p-1)}}{p!}h^{p} \Bigr)_{\! \xi}
 + O(|\xi^{-\mu(p+1)}|+|\xi^{-\mu p-1}|) \Bigr)d\xi
\\
=& 2\psi_0^2h +O(|h(x_{\nu})|+|x^{-2\mu}|+|x_{\nu}^{-2\mu}|),
\end{align*}
if $\mu(p-1)\ge 1,$ and that $\psi(x_{\nu}) -\psi_0(x_{\nu}) \ll h(x_{\nu})\psi
_0'(x_{\nu}),$ and using \eqref{3.82} and \eqref{3.4}, we have
\begin{align*}
\chi =& (4\psi'_0-2(A_{\phi}-\psi_0^2))h +\int^x_{\infty}(A_{\phi}-\psi_0^2)
(2F_2(\psi_0)-F_1(\psi_0,b_0)^2) \frac{d\xi}{\xi^2} +O(x^{-2\mu})
\\
 =& b'_0h +\int^x_{\infty}(A_{\phi}-\psi_0^2)
(2F_2(\psi_0)-F_1(\psi_0,b_0)^2) \frac{d\xi}{\xi^2} +O(x^{-2\mu}).
\end{align*}
Combining this with \eqref{3.81} and \eqref{3.9} we have the following.
\begin{prop}\label{prop3.2}
Under supposition \eqref{3.6} with $0<\mu \le 1,$ $h$ and $\chi=b-b_0$ 
satisfy
\begin{align*}
 &h= -\int^x_{\infty} F_1(\psi_0,b_0)\frac{d\xi}{\xi}
\\
&\phantom{--}  +\int^x_{\infty}
\Bigl(F_2(\psi_0)-\frac 32 F_1(\psi_0,b_0)^2 \Bigr)\frac{d\xi}{\xi^2}
+\frac 12 \int^x_{\infty} \frac{\chi-b_0'h}{A_{\phi}-\psi_0^2} \frac{d\xi}{\xi}
+O(x^{-2\mu}),
\\
&\chi-b'_0h= \int^x_{\infty} (A_{\phi}-\psi_0^2)(2F_2(\psi_0)-F_1(\psi_0,b_0)^2
)\frac{d\xi}{\xi^2} +O(x^{-2\mu})
\end{align*}
and $\chi \ll x^{-\mu}$
in $\check{S}_{\mathrm{cut}}(\phi,t_{\infty},\kappa_0,\delta_0)$, 
in which each integral converges. 
\end{prop}
\section{Proofs of the main theorems}\label{sc4}
Theorems \ref{thm2.1} and \ref{thm2.2} are immediately derived from the 
following proposition.
\begin{prop}\label{prop4.1}
Under supposition \eqref{3.6} with $0<\mu \le 1,$
$$
h(x)=-\frac{2((\theta_0-\theta_1)^2+\theta_{\infty}^2)}{A_{\phi}-1}x^{-1}
- \int^x_{\infty} F_1(\psi_0, b_0)\frac{d\xi}{\xi} -\frac 32 \int^x_{\infty}
F_1(\psi_0,b_0)^2\frac{d\xi}{\xi^2}+O(x^{-2\mu})
$$
in $\check{S}_{\mathrm{cut}}(\phi,t_{\infty},\kappa_0,\delta_0)$, where 
$$
\int^x_{\infty} F_1(\psi_0,b_0)\frac{d\xi}{\xi} \ll x^{-1}, \quad
\int^x_{\infty} F_1(\psi_0,b_0)^2\frac{d\xi}{\xi^2} \ll x^{-1}. 
$$
\end{prop}
{\bf Derivation of Theorems \ref{thm2.1} and \ref{thm2.2}.}
By Proposition \ref{prop3.1} or \cite[Theorem 2.1]{S1}, estimate 
\eqref{3.6} with $\mu=1/9$ is valid, and Proposition \ref{prop4.1}
with $\mu=1/9$ leads us to \eqref{3.6} with $h(x) \ll x^{-2/9}$ in
$\check{S}_{\mathrm{cut}}(\phi,t_{\infty},\kappa_0,\delta_0)$. 
Then Proposition \ref{prop4.1}
with $\mu=2/9$ yields the asymptotic formula for $h(x)$ with the error term
$O(x^{-4/9})$ and the estimate $h(x) \ll x^{-4/9}.$ Twice more repetition
of this procedure leads us to the desired asymptotic formula for 
$h(x)$ of Theorem \ref{thm2.2} in $\check{S}_{\mathrm{cut}}(\phi,t_{\infty},
\kappa_0,\delta_0).$ 
By Remark \ref{rem2.0}, in $\check{S}_{\mathrm{cut}}(\phi,t_{\infty},\kappa_0,
\delta_0)$, 
$$
(y(x)+1)(y(x)-1)^{-1} -A_{\phi}^{1/2} \mathrm{sn}((x-x_0)/2;A_{\phi}^{1/2}) 
\ll \psi_0'(x)h(x),
$$
where the left-hand side is holomorphic in $S(\phi_0,t_{\infty},\kappa_0,
\delta_0)$. 
By the maximal modulus principle, we have Theorem \ref{thm2.1}. 
\hfill$\square$
\begin{rem}\label{rem4.1}
By the argument above with \eqref{3.81} or Proposition \ref{prop3.2},  
in $S(\phi,t_{\infty},\kappa_0,\delta_0)$
$$
b(x)=b_0(x)+O(x^{-1}).
$$
\end{rem}
To complete the proofs of Theorems \ref{thm2.1} and \ref{thm2.2} it remains
to establish Proposition \ref{prop4.1}. 
The main part of the proof consists of evaluation of integrals, in which 
the following primitive functions are used \cite[Lemma 6.3]{S}.
\begin{lem}\label{lem4.2}
Let $\nu_0=(1+\tau_0)/2$ with $\tau_0=\Omega_{\mathbf{b}}/\Omega_{\mathbf{a}}$.
Then, for $\sn\, u=\sn (u;A_{\phi}^{1/2}),$
\begin{align*}
\int^u_0 \frac{du}{1-\sn^2u}=&\frac 1{(A_{\phi}-1)\Omega_{\mathbf{a}}}
\\
& \times
\Bigl(\mathcal{E}_{\mathbf{a}}u+\frac{\vartheta'}{\vartheta} \Bigl(\frac u
{\Omega_{\mathbf{a}}} -\frac 14+\nu_0, \tau_0 \Bigr)
+\frac{\vartheta'}{\vartheta} \Bigl(\frac u
{\Omega_{\mathbf{a}}} +\frac 14+\nu_0, \tau_0 \Bigr) +c_1 \Bigr),
\\
\int^u_0 \frac{\sn\, u\, du}{1-\sn^2u}=&\frac 1{(A_{\phi}-1)\Omega_{\mathbf{a}}}
\Bigl(\frac{\vartheta'}{\vartheta} \Bigl(\frac u
{\Omega_{\mathbf{a}}} -\frac 14+\nu_0, \tau_0 \Bigr)
-\frac{\vartheta'}{\vartheta} \Bigl(\frac u
{\Omega_{\mathbf{a}}} +\frac 14+\nu_0, \tau_0 \Bigr) +c_2 \Bigr),
\\
\int^u_0 \frac{du}{1- A_{\phi}\sn^2u}=&\frac 1{(1-A_{\phi})\Omega_{\mathbf{a}}}
\Bigl(\mathcal{E}_{\mathbf{a}} u
+\frac{\vartheta'}{\vartheta} \Bigl(\frac u
{\Omega_{\mathbf{a}}} -\frac 14, \tau_0 \Bigr)
+\frac{\vartheta'}{\vartheta} \Bigl(\frac u
{\Omega_{\mathbf{a}}} +\frac 14, \tau_0 \Bigr) \Bigr) +u,
\\
\int^u_0 \frac{\sn\, u\, du}{1- A_{\phi}\sn^2u}=&
\frac 1{A_{\phi}^{1/2}(1-A_{\phi})\Omega_{\mathbf{a}}}
\Bigl(\frac{\vartheta'}{\vartheta} \Bigl(\frac u
{\Omega_{\mathbf{a}}} +\frac 14, \tau_0 \Bigr)
-\frac{\vartheta'}{\vartheta} \Bigl(\frac u
{\Omega_{\mathbf{a}}} -\frac 14, \tau_0 \Bigr)+c_3 \Bigr), 
\\
\int^u_0 \frac{du}{(1-\sn^2u)^2}=& \frac{-1}{6(A_{\phi}-1)^2\Omega_{\mathbf{a}}}
\Bigl( \Bigl(\frac d{du}\Bigr)^{\! 2}+ { 4(1-2A_{\phi})} \Bigr) 
\Bigl(\mathcal{E}_{\mathbf{a}}u
+\frac{\vartheta'}{\vartheta} \Bigl(\frac u
{\Omega_{\mathbf{a}}} -\frac 14+\nu_0, \tau_0 \Bigr)
\\
& +\frac{\vartheta'}{\vartheta} \Bigl(\frac u
{\Omega_{\mathbf{a}}} +\frac 14+\nu_0, \tau_0 \Bigr) \Bigr)
-\frac {A_{\phi}}{3(A_{\phi}-1)}u,
\\
\int^u_0 \frac{\sn\, u\,du}{(1-\sn^2u)^2}=&
 \frac {-1}{6(A_{\phi}-1)^2\Omega_{\mathbf{a}}}
\Bigl( \Bigl(\frac d{du}\Bigr)^{\! 2} +1-5A_{\phi} \Bigr)
\Bigl(\frac{\vartheta'}{\vartheta} \Bigl(\frac u
{\Omega_{\mathbf{a}}} -\frac 14+\nu_0, \tau_0 \Bigr)
\\
& -\frac{\vartheta'}{\vartheta} \Bigl(\frac u
{\Omega_{\mathbf{a}}} +\frac 14+\nu_0, \tau_0 \Bigr) +c_4 \Bigr),
\end{align*}
where $c_j$ $(1\le j\le 4)$ are some constants.
\end{lem}
\begin{proof}
Recall the notation $4K=\Omega_{\mathbf{a}},$ $2iK'=\Omega_{\mathbf{b}}$
and $k=A_{\phi}^{1/2}.$ Observing the behaviours around the poles $u=2K\pm K$,
we have
$$
(k^2-1)(\sn^2u-1)^{-1} \equiv k^2(\sn^2(u-K+iK')-1),
$$
and 
\begin{align*}
 k^2(\sn^2(u-K & +iK')-1)
\\
+& \frac 1{\Omega_{\mathbf{a}}} \frac{d}{du}
\Bigl(\frac{\vartheta'}{\vartheta}
 \Bigl(\frac u{\Omega_{\mathbf{a}}} -\frac 14 +\nu_0,\tau_0\Bigr)+
\frac{\vartheta'}{\vartheta}
 \Bigl(\frac u{\Omega_{\mathbf{a}}} +\frac 14 +\nu_0,\tau_0\Bigr) \Bigr)
\equiv c_0
\end{align*}
(cf. \cite{H}, \cite{WW}). Integration on $[-iK',-iK'+K]$ yields 
$c_0=-\mathcal{E}_{\mathbf{a}}/\Omega_{\mathbf{a}},$ which implies 
the first formula. From
$$
(k^2-1)^2(\sn^2u-1)^{-2} \equiv k^4\cn^4(u-K+iK'),
$$
and
\begin{align*}
k^4&\cn^4 (u-K+iK')
\\
+& \Bigl(\frac 1{6\Omega_{\mathbf{a}}}\Bigl(\frac d{du}\Bigr)^{\!\! 3}+
\frac{2(1-2k^2)}{3\Omega_{\mathbf{a}}} \frac d{du} \Bigr)
\Bigl(\frac{\vartheta'}{\vartheta}
 \Bigl(\frac u{\Omega_{\mathbf{a}}} -\frac 14 +\nu_0,\tau_0\Bigr)+
\frac{\vartheta'}{\vartheta}
 \Bigl(\frac u{\Omega_{\mathbf{a}}} +\frac 14 +\nu_0,\tau_0\Bigr) \Bigr)
\equiv c_0
\end{align*}
with
$$
k^4\int^K_0 \cn^4u du = \frac 16(2k^2-1)\mathcal{E}_{\mathbf{a}} +
\frac{k^2}{12}(1-k^2)\Omega_{\mathbf{a}},
$$
the primitive function of $(\sn^2u-1)^{-2}$ follows.
\end{proof}
\subsection{Evaluation of integrals}\label{ssc4.1}
Write
$$
g(s)=\frac{\mathcal{E}_{\mathbf{a}}}2 s +\frac{\vartheta'}{\vartheta}
\Bigl(\frac s{\Omega_{\mathbf{a}}}, \tau_0 \Bigr), 
$$
which is bounded for $2s+x_0 \in {S}(\phi, t_{\infty}, \kappa_0,\delta_0)$ and
satisfies $g((x-x_0)/2)=\mathfrak{b}(x).$ Then, by Lemma \ref{lem4.2}, 
\begin{align*}
& \int^s_{\infty} \frac{\sn\,\sigma}{1-\sn^2\sigma}\frac{d\sigma}{\tilde{\sigma}}
 =\frac 1{(A_{\phi}-1)\Omega_{\mathbf{a}}} \int^s_{\infty} \Bigl(\frac
{\vartheta'}{\vartheta}\Bigl(\frac{\sigma-\alpha_0}{\Omega_{\mathbf{a}}}\Bigr)
- \frac
{\vartheta'}{\vartheta}\Bigl(\frac{\sigma+\alpha_0}{\Omega_{\mathbf{a}}}\Bigr)
\Bigr)_{\! \sigma} \frac{d\sigma}{\tilde{\sigma}}
\\
& \ll \biggl| \Bigl(\frac{\vartheta'}{\vartheta}
\Bigl(\frac{s-\alpha_0}{\Omega_{\mathbf{a}}}\Bigr)
- \frac{\vartheta'}{\vartheta}
\Bigl(\frac{s+\alpha_0}{\Omega_{\mathbf{a}}}\Bigr)
\Bigr)\frac 1{\tilde{s}} \biggr| +
\biggl| \int^s_{\infty} \Bigl(\frac{\vartheta'}{\vartheta}
\Bigl(\frac{\sigma-\alpha_0}{\Omega_{\mathbf{a}}}\Bigr)
- \frac{\vartheta'}{\vartheta}
\Bigl(\frac{\sigma+\alpha_0}{\Omega_{\mathbf{a}}}\Bigr)
\Bigr)\frac {d\sigma}{\tilde{\sigma}^2}  \biggr| \ll s^{-1} 
\end{align*}
and
\begin{align*}
\int^s_{\infty}& \frac{1}{1-\sn^2\sigma}  \frac{d\sigma}{\tilde{\sigma}}
 =\frac 1{(A_{\phi}-1)\Omega_{\mathbf{a}}} \int^s_{\infty} \Bigl(\mathcal{E}
_{\mathbf{a}} \sigma +\frac{\vartheta'}{\vartheta}
\Bigl(\frac{\sigma-\alpha_0}{\Omega_{\mathbf{a}}}\Bigr)
+ \frac{\vartheta'}{\vartheta}
\Bigl(\frac{\sigma+\alpha_0}{\Omega_{\mathbf{a}}}\Bigr)
\Bigr)_{\! \sigma} \frac{d\sigma}{\tilde{\sigma}}
\\
& =\frac 1{(A_{\phi}-1)\Omega_{\mathbf{a}}} \int^s_{\infty} (g(\sigma-\alpha_0)
+g(\sigma+\alpha_0))_{\sigma}  \frac {d\sigma}{\tilde{\sigma}}
\\
& \ll \Bigl|(g(s-\alpha_0)+g(s+\alpha_0) )\frac 1{\tilde{s}}
\Bigr| + \biggl| \int^s_{\infty} (g(\sigma-\alpha_0)
+g(\sigma+\alpha_0)) \frac {d\sigma}{\tilde{\sigma}^2}\biggr| \ll s^{-1} 
\end{align*}
with $\tilde{\sigma}=\sigma + x_0/2,$ $\tilde{s}=s+x_0/2$ and
$\alpha_0=(1/4+\nu_0)\Omega_{\mathbf{a}}$, which also implies 
the convergence of these integrals. Then we may write 
\begin{align}\notag
& \int^x_{\infty} F_1(\psi_0,b_0)\frac{d\xi}{\xi} = \int^x_{\infty}\frac
{(4(\theta_0+\theta_1)\psi_0-b_0)}{2(A_{\phi}-\psi_0^2)}\frac{d\xi}{\xi}
\\
\notag
=& 2(\theta_0+\theta_1)\int^x_{\infty}\frac{\psi_0}{A_{\phi}-\psi_0^2}\frac{d\xi}
{\xi}
-\frac {b_0(x_0)}2 \int^x_{\infty}\frac 1{A_{\phi}-\psi_0^2}\frac{d\xi}{\xi}
+ \frac 4{\Omega_{\mathbf{a}}} \int^x_{\infty}\frac{\mathfrak{b}(\xi)}{A_{\phi}
-\psi_0^2} \frac{d\xi}{\xi}
\\
\label{4.1}
=&\frac{2(\theta_0+\theta_1)}{A_{\phi}^{1/2}} \int^s_{\infty} \frac{\sn\,\sigma}
{1-\sn^2\sigma} \frac{d\sigma}{\tilde{\sigma}}
-\frac {b_0(x_0)}{2A_{\phi}} 
\int^s_{\infty}\frac 1{1-\sn^2\sigma}\frac{d\sigma}{\tilde
{\sigma}} +\frac{4}{A_{\phi}\Omega_{\mathbf{a}}} \int^s_{\infty} \frac
{g(\sigma)}{1-\sn^2\sigma} \frac{d\sigma}{\tilde{\sigma}}, 
\end{align}
with $\sigma=(\xi-x_0)/2$, $ s=(x-x_0)/2,$ $b_0(x_0)=\beta_0-{2\mathcal{E}
_{\mathbf{a}}}\Omega_{\mathbf{a}}^{-1} x_0$. 
On the last two lines of \eqref{4.1} the first two integrals converge,
and consequently, by Proposition \ref{prop3.2}, so the integral
containing $\mathfrak{b}(\xi)$ or $g(\sigma)$. Let us evaluate it. Set
\begin{equation*}
 \mathcal{J}_0:=
(A_{\phi}-1)\Omega_{\mathbf{a}} \int^s_{\infty} \frac{g(\sigma)}{1-\sn^2\sigma}
\frac{d\sigma}{\tilde{\sigma}} = \int^s_{\infty} (g(\sigma-\alpha_0)
+g(\sigma+\alpha_0) )_{\sigma} g(\sigma) \frac{d\sigma}{\tilde{\sigma}}.
\end{equation*}
For any sequence $\{s_{\nu}\}$ with $s_{\nu}=(x_{\nu}-x_0)/2,$  
\begin{align*}
 \int^s_{s_{\nu}} g_{\sigma}(\sigma+\alpha_0) g(\sigma)\frac{d\sigma}{\tilde
{\sigma}}=& g(\sigma+\alpha_0)g(\sigma) \tilde{\sigma}^{-1} \Bigr]^s_{s_{\nu}}
\\
&- \int^s_{s_{\nu}} g(\sigma+\alpha_0) g_{\sigma}(\sigma)
\frac{d\sigma}{\tilde{\sigma}}
+ \int^s_{s_{\nu}} g(\sigma+\alpha_0) g(\sigma)
\frac{d\sigma}{\tilde{\sigma}^2}
\\
=& -\int^{s+\alpha_0}_{s_{\nu}+\alpha_0}
 g(\rho) g_{\rho}(\rho-\alpha_0)\frac{d\rho}{\tilde{\rho}-\alpha_0}
+O(s^{-1})+O(s^{-1}_{\nu}) 
\\
=& -\int^s_{s_{\nu}} g_{\sigma}(\sigma-\alpha_0)g(\sigma)
 \frac{d\sigma}{\tilde{\sigma}}+O(s^{-1})+O(s^{-1}_{\nu}), 
\end{align*}
which implies $\mathcal{J}_0 \ll s^{-1}.$
Thus we have the following crucial estimate. 
\begin{prop}\label{prop4.3}
In $\check{S}_{\mathrm{cut}}(\phi,t_{\infty},\kappa_0,\delta_0)$,
\begin{align*}
\int^x_{\infty}F_1(\psi_0,b_0)\frac{d\xi}{\xi} =& 2(\theta_0+\theta_1)
\int^x_{\infty} \frac{\psi_0}{A_{\phi}-\psi_0^2}\frac{d\xi}{\xi}
\\
& -\frac{b_0(x_0)}{2}
\int^x_{\infty} \frac{1}{A_{\phi}-\psi_0^2}\frac{d\xi}{\xi}
+\frac 4{\Omega_{\mathbf{a}}}
\int^x_{\infty} \frac{\mathfrak{b}(\xi)}{A_{\phi}-\psi_0^2}\frac{d\xi}{\xi}
\ll x^{-1},
\end{align*}
where each integral on the right-hand side is $O(x^{-1}).$
\end{prop}
Observe that
\begin{align*}
\int^x_{\infty} F_2(\psi_0) \frac{d\xi}{\xi^2}
=& \int^x_{\infty} \frac{2( 2(\theta_0-\theta_1)\theta_{\infty} \psi_0 +
(\theta_0-\theta_1)^2+\theta_{\infty}^2 )}{(1-\psi_0^2)(A_{\phi}-\psi_0^2)}
\frac{d\xi}{\xi^2}
\\
=&\frac 1{A_{\phi}(A_{\phi}-1)} \biggl( 2(\theta_0-\theta_1)\theta_{\infty}
A_{\phi}^{1/2} \int^s_{\infty} \Bigl(\frac{A_{\phi} \sn\, \sigma}{1-A_{\phi}\sn^2
\sigma}-\frac{\sn\,\sigma}{1-\sn^2\sigma}\Bigr) \frac{d\sigma}{\tilde{\sigma}^2} 
\\
& + ((\theta_0-\theta_1)^2+\theta_{\infty}^2) \int^s_{\infty} 
\Bigl(\frac{A_{\phi}}{1-A_{\phi}\sn^2\sigma} -\frac 1{1-\sn^2\sigma} \Bigr)
\frac{d\sigma}{\tilde{\sigma}^2} \biggr). 
\end{align*}
In the last line
\begin{align*}
 \int^s_{\infty} 
\frac{A_{\phi}}{1-A_{\phi}\sn^2\sigma} 
\frac{d\sigma}{\tilde{\sigma}^2}  
=& {A_{\phi}} \int^s_{\infty} \Bigl(
\frac{g(\sigma-\Omega_{\mathbf{a}}/4)
 +g(\sigma+ \Omega_{\mathbf{a}}/4)}{(1-A_{\phi})\Omega_{\mathbf{a}} } +\sigma
\Bigr)_{\! \sigma} \frac{d\sigma}
{\tilde{\sigma}^2}
\\
=& -A_{\phi}s^{-1} +O(s^{-2}),
\end{align*}
and the remaining three integrals are $O(s^{-2}).$ Thus we have the following.
\begin{prop}\label{prop4.4}
In $\check{S}_{\mathrm{cut}}(\phi, t_{\infty},\kappa_0,\delta_0)$,
$$
\int^x_{\infty} F_2(\psi_0) \frac{d\xi}{\xi^2} = -\frac{2((\theta_0-\theta_1)^2
+\theta_{\infty}^2)}{A_{\phi}-1} x^{-1} +O(x^{-2}).
$$
\end{prop}
\subsection{Proof of Proposition \ref{prop4.1}}\label{ssc4.2}
By Proposition \ref{prop3.2}, we have
\begin{equation}\label{4.2}
h(x)= -\int^x_{\infty} F_1(\psi_0,b_0)\frac{d\xi}{\xi} +\int^x_{\infty}
\Bigl(F_2(\psi_0)-\frac 32 F_1(\psi_0,b_0)^2 \Bigr) \frac{d\xi}{\xi^2}
+\frac 12 \mathcal{J}_1 + O(x^{-2\mu})
\end{equation}
with
$$
\mathcal{J}_1 =\int^x_{\infty} \frac 1{A_{\phi}-\psi_0^2} \int^{\xi}_{\infty}
(A_{\phi}-\psi_0^2) (2F_2(\psi_0)-F_1(\psi_0,b_0)^2 ) \frac{d\xi_1}{\xi_1^2}
\frac{d\xi}{\xi}.
$$
Then 
\begin{align*}
\mathcal{J}_1 =& \int^x_{\infty} \frac 1{A_{\phi}-\psi_0^2} 
\frac{d\xi}{\xi} \cdot  \int^{x}_{\infty}
(A_{\phi}-\psi_0^2) (2F_2(\psi_0)-F_1(\psi_0,b_0)^2 ) \frac{d\xi_1}{\xi_1^2}
\\
& - \int^x_{\infty} \int^{\xi}_{\infty} \frac 1{A_{\phi}-\psi_0^2} 
\frac{d\xi_1}{\xi_1} \cdot  
(A_{\phi}-\psi_0^2) (2F_2(\psi_0)-F_1(\psi_0,b_0)^2 ) \frac{d\xi}{\xi^2}
\ll x^{-2},
\end{align*}
since $\int^x_{\infty} (A_{\phi}-\psi_0^2)^{-1} \xi^{-1} d\xi \ll x^{-1}.$
Insertion of $\mathcal{J}_1$ into \eqref{4.2} combined with Propositions
\ref{prop4.3} and \ref{prop4.4} yields the desired expression of $h(x).$
Thus we have Proposition \ref{prop4.1}.
\subsection{Further calculation of integrals for $h(x)$}\label{ssc4.3}
In the expression of $h(x)$ in Proposition \ref{prop4.1}, the second
integral becomes
\begin{align}\notag
\int^x_{\infty} & F_1(\psi_0,b_0)^2 \frac{d\xi}{\xi^2} =
 \frac 2{A_{\phi}} (\theta_0+\theta_1)^2 \int^s_{\infty} \Bigl(\frac 1
{(1-\sn^2\sigma)^2}-\frac 1{1-\sn^2\sigma} \Bigr) \frac{d\sigma}
{\tilde{\sigma}^2}
\\
\notag
&- \frac{(\theta_0+\theta_1)}{A_{\phi}^{3/2}} \int^s_{\infty} 
\frac{(b_0(x_0)-8\Omega^{-1}_{\mathbf{a}} g(\sigma) )\sn\,\sigma }
{(1-\sn^2\sigma)^2} \frac{d\sigma}{\tilde{\sigma}^2}
\\
\notag
&+ \frac 1{8A_{\phi}^2} \int^s_{\infty} \frac{b_0(x_0)^2 -16\Omega^{-1}
_{\mathbf{a}} b_0(x_0) g(\sigma) +64\Omega^{-2}_{\mathbf{a}} g(\sigma)^2)}
{(1-\sn^2\sigma)^2} \frac{d\sigma}{\tilde{\sigma}^2}
\\
\notag
=& \frac{4(\theta_0+\theta_1)^2}{3(A_{\phi}-1)}x^{-1} +\frac{b_0(x_0)^2}
{12A_{\phi}(A_{\phi}-1)}x^{-1} +\frac{4b_0(x_0)}{3A_{\phi}(A_{\phi}-1)\Omega
_{\mathbf{a}} }\int^x_{\infty} \mathfrak{b}(\xi) \frac{d\xi}{\xi^2}
\\
\label{4.3}
&+ \frac{16(\theta_0+\theta_1)}{\Omega_{\mathbf{a}}} \int^x_{\infty} 
\frac{\mathfrak{b}(\xi) \psi_0}{(A_{\phi}-\psi_0^2)^2} \frac{d\xi}{\xi^2}
+ \frac{16}{\Omega_{\mathbf{a}}^2} \int^x_{\infty} 
\frac{\mathfrak{b}(\xi)^2 }{(A_{\phi}-\psi_0^2)^2} \frac{d\xi}{\xi^2}
+O(x^{-2}).
\end{align}
This is obtained by using
\begin{align*}
&\int^s_{\infty}\frac 1{(1-\sn^2\sigma)^2}\frac{d\sigma}{\tilde{\sigma}^2} 
=\frac{A_{\phi}}{3(A_{\phi}-1)}s^{-1}+O(s^{-2}), 
\\
&\int^s_{\infty}\frac 1{1-\sn^2\sigma}\frac{d\sigma}{\tilde{\sigma}^2} \ll
s^{-2}, \quad 
\int^s_{\infty}\frac {\sn\,\sigma}
{(1-\sn^2\sigma)^2}\frac{d\sigma}{\tilde{\sigma}^2} \ll s^{-2}
\end{align*}
and
$$
\int^s_{\infty} \frac{g(\sigma)}{(1-\sn^2\sigma)^2} \frac{d\sigma}{\tilde{\sigma}
^2} = -\frac{A_{\phi}}{3(A_{\phi}-1)} \int^s_{\infty} g(\sigma) \frac{d\sigma}
{\tilde{\sigma}^2} +O(s^{-2}).
$$
In deriving the last equality we note the following:
$$
\int^s_{\infty}(g(\sigma+\alpha_0)+g(\sigma-\alpha_0))_{\sigma}g(\sigma)\frac
{d\sigma}{\tilde{\sigma}^2}, \,\,\,
\int^s_{\infty}(g_{\sigma}(\sigma+\alpha_0)
+g_{\sigma}(\sigma-\alpha_0))_{\sigma}g_{\sigma} (\sigma)\frac
{d\sigma}{\tilde{\sigma}^2} \ll s^{-2}, 
$$
which are shown by the same way as in the proof of $\mathcal{J}_0 \ll s^{-1}$
in Section \ref{ssc4.1}.
By \eqref{4.1} and \eqref{4.3}, $h(x)$ is written in the form
\begin{align*}
h(x) =& -\frac{2(2\theta_0^2+2\theta_1^2 +\theta_{\infty}^2)}{A_{\phi}-1}x^{-1}
-2(\theta_0+\theta_1) \int^x_{\infty} \frac{\psi_0}{A_{\phi}-\psi_0^2} \frac
{d\xi}{\xi} 
\\
& + \frac {b_0(x_0)}2 
 \int^x_{\infty} \frac 1{A_{\phi}-\psi_0^2} \frac{d\xi}{\xi}
-  \frac{4}{\Omega_{\mathbf{a}}} \int^x_{\infty} \frac{\mathfrak{b}(\xi)}
{A_{\phi}-\psi_0^2} \frac{d\xi}{\xi}
\\
& -\frac{b_0(x_0)^2}{8A_{\phi}(A_{\phi}-1)}x^{-1} -\frac{2b_0(x_0)}{A_{\phi}
(A_{\phi}-1)\Omega_{\mathbf{a}}} \int^x_{\infty} \mathfrak{b}(\xi) \frac{d\xi}
{\xi^2}
\\
& -\frac{24}{\Omega_{\mathbf{a}}} (\theta_0+\theta_1) \int^x_{\infty}
\frac{\mathfrak{b}(\xi)\psi_0}{(A_{\phi}-\psi_0^2)^2} \frac{d\xi}{\xi^2}
 -\frac{24}{\Omega_{\mathbf{a}}^2}  \int^x_{\infty}
\frac{\mathfrak{b}(\xi)^2}{(A_{\phi}-\psi_0^2)^2} \frac{d\xi}{\xi^2}
+O(x^{-2}).
\end{align*}
\subsection{Proof of Theorem \ref{thm2.3}}\label{ssc4.4}
Recalling Remark \ref{rem4.1} and combining
$$
\int^x_{\infty} (A_{\phi}-\psi_0^2)F_2(\psi_0)\frac{d\xi}{\xi^2} =
-2((\theta_0-\theta_1)^2+\theta_{\infty}^2)x^{-1}+O(x^{-2})
$$
with the second equality of Proposition \ref{prop3.2}, we obtain Theorem 
\ref{thm2.3} by the same argument as in the derivation  of Theorems \ref{thm2.1}
and \ref{thm2.2}. Furthermore we have
\begin{align*}
b(x)= & b_0(x)+b'_0(x)h(x) -4(2\theta_0^2 +2\theta_1^2 +\theta_{\infty}^2)x^{-1}
\\
&- \frac{16(\theta_0+\theta_1)}{\Omega_{\mathbf{a}}} \int^x_{\infty}
\frac{ \mathfrak{b}(\xi) \psi_0}{A_{\phi}-\psi_0^2} \frac{d\xi}{\xi^2}
-\frac{16}{\Omega_{\mathbf{a}}^2} \int^x_{\infty} \frac{\mathfrak{b}(\xi)^2}
{A_{\phi}-\psi_0^2} \frac{d\xi}{\xi^2} +O(x^{-2})
\end{align*}
as in Remark \ref{rem2.1}.


\end{document}